\documentclass[11pt]{article}

\usepackage{amsmath, amsthm, amsfonts, graphicx, fancyhdr, amssymb, epstopdf}

\author{Jay Schweig \footnote{University of Kansas; jschweig@math.ku.edu}}
\title{Convex-Ear Decompositions and the Flag h-Vector}
\newtheorem{thm}{Theorem}[section]
\newtheorem{prop}[thm]{Proposition}
\newtheorem{cor}[thm]{Corollary}
\newtheorem{lem}[thm]{Lemma}
\newtheorem{defn}[thm]{Definition}
\newtheorem{fact}[thm]{Fact}
\newtheorem{conj}[thm]{Conjecture}

\newtheorem{example}[thm]{Example}

\newcommand{\old}{\text{old}}
\newcommand{\new}{\text{new}}
\newcommand{\bc}{\textbf{c}}
\newcommand{\bd}{\textbf{d}}
\newcommand{\be}{\textbf{e}}

\newcommand{\ssection}[1]{%
  \section[#1]{\centering\normalfont\scshape #1}}
\newcommand{\ssubsection}[1]{%
  \subsection[#1]{\raggedright\normalfont\itshape #1}}
                                                                                  
\begin{document} 

\maketitle

\begin{abstract}
We prove a theorem allowing us to find convex-ear decompositions for rank-selected subposets of posets that are unions of Boolean sublattices in a coherent fashion.  We then apply this theorem to geometric lattices and face posets of shellable complexes, obtaining new inequalities for their h-vectors.  Finally, we use the latter decomposition to prove new inequalities for the flag h-vectors of face posets of Cohen-Macaulay complexes. 
\end{abstract}

\ssection{Introduction}
The \emph{f-vector} of a finite simplicial complex $\Delta$, which counts the number of faces of the complex in each dimension, is arguably its most fundamental invariant.  The \emph{h-vector} of $\Delta$ is the image of its f-vector under an invertible transformation.  Somewhat surprisingly, properties of a complex's f-vector are sometimes better expressed through its h-vector.  A good example of this phenomenon are the Dehn-Sommerville relations (see, for instance, \cite{ziegler}), which state that the h-vector of a simplicial polytope boundary is symmetric.  

The main complexes we study in this paper are all \emph{order complexes}, namely complexes whose simplices correspond to chains in posets.  Since a poset and its order complex hold the same information, we often refer to them interchangeably.  E.g., we may speak of the \emph{facets} or \emph{h-vector} of a poset, or to a \emph{chain} in an order complex. 

\emph{Convex-ear decompositions} were first introduced by Chari in \cite{ch}.  Heuristically, a complex admits a convex-ear decomposition if it is a union of simplicial polytope boundaries which fit together coherently (see Definition \ref{ced}).  Suppose a $(d-1)$-dimensional complex $\Delta$ admits such a decomposition.  In \cite{ch}, Chari shows that the h-vector $(h_0, h_1, \ldots, h_d)$ of $\Delta$ satisfies, for $i < d/2$, $h_i \leq h_{i+1}$ and $h_i \leq h_{d-i}$.  In \cite{ed}, Swartz shows that $\Delta$ is $2$-CM, and that $(h_0, h_1-h_0, h_2-h_1, \ldots, h_{\lfloor d/2\rfloor} - h_{\lfloor d/2\rfloor - 1})$ is an M-vector (called an \emph{O-sequence} by some authors).  Convex-ear decompositions have proven quite useful, as they have been applied to coloop-free matroid complexes \cite{ch}, geometric lattices \cite{ns}, coloring complexes \cite{edandtrisha}, $d$-divisible partition lattices \cite{russ}, coset lattices of relatively complemented finite groups \cite{russ}, and finite buildings \cite{ed}.  

In \cite{me}, we find a convex-ear decomposition rank-selected subposets of supersolvable lattices with nowhere-zero M\"obius functions.  In the process, we obtain a decomposition for order complexes of rank-selected subposets of the Boolean lattice $B_d$.  In this paper we build upon this result with Theorem \ref{boolpieces}, which gives convex-ear decompositions for rank-selected subposets of posets that are unions of Boolean lattices, pieced together nicely.  In Section \ref{central} we recall several useful results from \cite{me}, and then prove this theorem.

Sections \ref{geolattice} and \ref{faceposetsection} are devoted to the applications of Theorem \ref{boolpieces} to geometric lattices and face posets of shellable complexes, respectively.  Taken together, these results (along with those of Chari and Swartz) give us the following.

\begin{cor}
Let $\Delta$ be the order complex of a rank-selected subposet of either
\begin{itemize}
\item[1:] a geometric lattice, or
\item[2:] the face poset of the codimension-$1$ skeleton of a shellable complex.
\end{itemize}
Then $\Delta$ is $2$-CM and its h-vector $(h_0, h_1, \ldots, h_d)$ satisfies, for all $i < d/2$, $h_i \leq h_{i+1}$ and $h_i \leq h_{d-i}$.  Moreover, $(h_0, h_1-h_0, h_2-h_1, \ldots, h_{\lfloor d/2\rfloor} - h_{\lfloor d/2\rfloor - 1})$ is an M-vector.  
\end{cor}

Finally, in Section \ref{flaghsection}, we use the decomposition from Section \ref{faceposetsection} and techniques similar to those in \cite{ns} to prove that the flag h-vector $\{h_S\}$ of a Cohen-Macaulay complex's face poset satisfies $h_T \leq h_S$ whenever $S$ dominates $T$ (in the sense of Definition \ref{dominance}).  

\ssection{Preliminaries}

We assume a familiarity with simplicial complexes and partially ordered sets (see \cite{ec1}).  All our simplicial complexes will be finite and pure.  

The \emph{f-vector} of a $(d-1)$-dimensional simplicial complex $\Delta$ is the sequence $(f_0, f_1, \ldots, f_d)$, where $f_i$ counts the number of $(i-1)$-dimensional faces of $\Delta$.  The \emph{h-vector} of $\Delta$ is the sequence $(h_0, h_1, \ldots, h_d)$ given by 
\[
\sum_{i=0}^d f_i(t-1)^{d-i} = \sum_{i=0}^d h_it^{d-i}
\]

We use the following alternate definition of shellability, easily seen to be equivalent to the standard one (see \cite{bj}).  

\begin{defn}\label{altshell}
Let $F_1, F_2, \ldots, F_t$ be an ordering of the facets of $\Delta$.  This ordering is a shelling if and only if for all $j < k$ there exists a $j' < k$ satisfying
\[
F_j\cap F_k \subseteq F_{j'} \cap F_k = F_{j'} - x
\]
for some element $x$ of $F_{j'}$.  
\end{defn}

We also use the following result of Danaraj and Klee for showing that a given complex is a ball.

\begin{thm}[\cite{dk}]\label{danarajklee}
Let $\Delta$ be a full-dimensional shellable proper subcomplex of a sphere.  Then $\Delta$ is a ball.
\end{thm}

\ssubsection{Convex-ear decompositions}

\begin{defn}\label{ced}
Let $\Delta$ be a $(d-1)$-dimensional complex.  We say $\Delta$ admits a \emph{convex-ear decomposition}, or \emph{c.e.d.}, if there exists a sequence of pure $(d-1)$-dimensional subcomplexes $\Sigma_1, \Sigma_2, \ldots, \Sigma_t$ such that
\begin{itemize}
\item[i:] $\bigcup_{i=1}^t\Sigma_i = \Delta$.
\item[ii:] For $i >1$, $\Sigma_i$ is a proper subcomplex of the boundary complex of a $d$-dimensional simplicial polytope, while $\Sigma_1$ is the boundary complex of a $d$-dimensional simplicial polytope.
\item[iii:] Each $\Sigma_i$, for $i > 1$ is a topological ball.
\item[iv:] For $i > 1$, $\Sigma_i\cap (\bigcup_{j=1}^{i - 1} \Sigma_j) = \partial \Sigma_i$.  
\end{itemize}
\end{defn}

Convex-ear decompositions were introduced by Chari in \cite{ch}, where the following was proven.

\begin{thm}\label{hced}
When $\Delta$ admits a convex-ear decomposition its h-vector satisfies, for all $i \leq \lfloor d/2\rfloor$,
\begin{itemize}
\item[1:] $h_i \leq h_{i+1}$ and
\item[2:] $h_i \leq h_{d-i}$. 
\end{itemize}
\end{thm}

In \cite{ed}, Swartz proved the following analogue of the g-Theorem for complexes with convex-ear decompositions.  An \emph{M-vector} is the degree sequence of an order ideal of monomials.  

\begin{thm}
Let $\Delta$ be a complex admitting a convex-ear decomposition, with h-vector $(h_0, h_1, \ldots, h_d)$.  Then the vector 
\[
(h_0, h_1 - h_0, h_2 - h_1, \ldots, h_{\lfloor d/2 \rfloor} - h_{\lfloor d/2 \rfloor -1} )
\]
is an M-vector.  Furthermore, $\Delta$ is $2$-CM.  
\end{thm}

\ssubsection{Order complexes and flag vectors}

Recall that the \emph{order complex} of a poset $P$, which we write as $\Delta(P)$, is the simplicial complex whose faces are chains in $P$.  If $P$ has a unique minimal element $\hat{0}$ or a unique maximal element $\hat{1}$, we do not include these in the order complex.  That is, simplices in $\Delta(P)$ are chains in $P - \{\hat{0}, \hat{1}\}$.  All our posets are ranked. 

For the remainder of this section, let $P$ be a rank-$d$ poset with a $\hat{0}$ and $\hat{1}$.  A \emph{labeling} of $P$ is a function $\lambda: \{(x, y)\in P^2 : y \text{ covers } x\} \rightarrow \mathbb{Z}$.  For a saturated chain
\[
\bc: = x = x_0 < x_1 < \cdots < x_t = y
\]
the $\lambda$-label of $\bc$, written $\lambda(\bc)$, is the word
\[
\lambda(x_0, x_1)\lambda(x_1, x_2) \cdots \lambda(x_{t-1}, x_t).
\]

\begin{defn}
A labeling $\lambda$ of $P$ is an \emph{EL-labeling} if:
\begin{itemize}
\item[1:] in each interval $[x ,y]$ in $P$, there is a unique saturated chain with a strictly increasing label, and
\item[2:] the label of this chain is lexicographically first among the labels of all saturated chains in $[x, y]$.
\end{itemize} 
If $\lambda$ is an EL-labeling of $P$ and each maximal chain is labeled with a permutation of $[d]$ (that is, an element of $\mathcal{S}_d$) then $\lambda$ is called an \emph{$\mathcal{S}_d$-EL-labeling}.
\end{defn}

\begin{example}
The Boolean lattice $B_d$ admits an $\mathcal{S}_d$-EL-labeling in an obvious way: if $y$ covers $x$ then $y = x \cup \{i\}$, so set $\lambda(x, y) = i$.  
\end{example}

When $P$ admits an EL-labeling $\lambda$ and $\bc$ is a chain in $P$, we write $\Upsilon_\lambda(\bc)$ to denote the maximal chain of $P$ obtained by filling in each gap in $\bc$ with the unique chain in that interval with increasing $\lambda$-label.  EL-labelings were introduced by Bj\"orner and Wachs, where the following was shown.

\begin{thm}[\cite{bw}]
If $P$ admits an EL-labeling then $\Delta(P)$ is shellable.
\end{thm}

For any $S\subseteq [d-1]$ and any maximal chain 
\[
\bc : = \hat{0} = x_0 < x_1 < x_2 < \cdots < x_d = \hat{1}
\]
of $P$, let $\bc_S$ denote the chain of elements of $\bc$ whose ranks lie in $S\cup \{0, d\}$.  The \emph{rank-selected subposet} $P_S$ is the subposet of $P$ whose maximal chains are all of the form $\bc_S$, as $\bc$ ranges over all maximal chains of $P$.  Equivalently, $P_S$ is the poset $P$ restricted to all elements with ranks in $S\cup\{0, d\}$.  

For any $S \subseteq [d-1]$, let $f_S$ be the number of maximal chains in $P_S$.  The collection $\{f_S\}$ is known as the \emph{flag f-vector} of $P$.  Note that the flag f-vector of $P$ refines its f-vector, as clearly
\[
f_i(P) = \sum_{S\subseteq [d-1], |S| = i} f_S(P).
\]
The \emph{flag h-vector} of $P$ is the collection $\{h_S\}$ defined by 
\[
h_S = \sum_{T\subseteq S} (-1)^{|S - T|}f_T.
\]
By inclusion-exclusion, the above is equivalent to $f_T = \sum_{S\subseteq T} h_S$.  It follows that the h-vector is refined by the flag h-vector, namely
\[
h_i(P) = \sum_{S \subseteq [d-1], |S| = i} h_S(P).
\]
When $P$ has an EL-labeling, its flag h-vector has a nice enumerative interpretation.

\begin{thm}[\cite{bw}]
The flag h-vector $\{h_S\}$ of a poset $P$ with EL-labeling $\lambda$ is given as follows: $h_S$ counts the number of maximal chains of $P$ whose $\lambda$-labels have descent set $S$.  
\end{thm}

\ssubsection{Dominance in $\mathcal{S}_d$.}

Let $\sigma$ be a permutation in the symmetric group $\mathcal{S}_d$.  We view $\sigma$ as a word in $[d]$, writing $\sigma = \sigma(1)\sigma(2)\cdots \sigma(d)$.  If $\sigma(i) < \sigma(i+1)$ call the interchanging of $\sigma(i)$ and $\sigma(i+1)$ in $\sigma$ a \emph{switch}.  

Recall that the \emph{weak order} on $\mathcal{S}_d$, for which we write $<_w$, is the partial order given by the following property:  $\sigma <_w \tau$ if and only if $\tau$ can be obtained from $\sigma$ by a sequence of switches.  

For $S \subseteq [d-1]$, let $D_S^d$ denote the set of permutation in $\mathcal{S}_d$ whose descent sets equal $S$:
\[
\sigma\in D_S^d \Leftrightarrow \{i: \sigma(i) > \sigma(i+1)\} = S.
\]

\begin{defn}\label{dominance}
Let $S, T \subseteq [d-1]$.  We say that $S$ \emph{dominates} $T$ if there exists an injection $\phi: D_T^d \rightarrow D_S^d$ such that $\sigma <_w \phi(\sigma)$ for all $\sigma\in D_T^d$.
\end{defn}

For example, let $d = 4$.  Then the set $\{1, 3\}$ dominates the set $\{1\}$ via the map 
\[
\sigma(1)\sigma(2)\sigma(3)\sigma(4) \rightarrow \sigma(1)\sigma(2)\sigma(4)\sigma(3).
\]

For a further discussion of dominance in the symmetric group, see \cite{ns} or \cite{de}.  

\ssection{A decomposition theorem}\label{central}

The goal of this section is to prove the following theorem, which we will apply in Sections \ref{geolattice} and \ref{faceposetsection}. 

\begin{thm}\label{boolpieces}
Let $P$ be a rank-$d$ poset with a $\hat{0}$ and a $\hat{1}$, and suppose $P_1, P_2, \ldots, P_r$ are subposets of $P$ satisfying the following properties.
\begin{itemize}
\item[1:] Each $P_i$ is isomorphic to the Boolean lattice $B_d$.
\item[2:] Every chain in $P$ is a chain in some $P_i$.  Equivalently, 
\[
\Delta(P) = \bigcup_{i=1}^r \Delta(P_i).
\]
\item[3:] Each $P_i$ has an $\mathcal{S}_d$-EL-labeling $\lambda_i$ with the following property: if $\textbf{e}$ is a chain in $P_i$ that is also a chain in some $P_j$ for $j <i$, then $\Upsilon_{\lambda_i}(\textbf{e})$ is a chain in $P_j$ for some $j < i$.  
\end{itemize}
Then $\Delta(P_S)$ admits a convex-ear decomposition for any $S \subseteq [d-1]$.
\end{thm}

The proof of this theorem relies heavily on our work from \cite{me}, where a c.e.d. of $\Delta((B_d)_S)$ was given for any $S \subseteq [d-1]$.  We now review some results from \cite{me} which will be helpful in proving Theorem \ref{boolpieces}.      

Fix some $S \subseteq [d-1]$ and an $\mathcal{S}_d$-EL-labeling $\lambda$ of $B_d$.  Let $\bd_1, \bd_2, \ldots, \bd_t$ be all maximal chains in $B_d$ whose $\lambda$-labels have descent set $S$, written in lexicographic order of their $\lambda$-labels.  For each $i$, let $L_i$ be the subposet of $(B_d)_S$ generated by the set of maximal chains
\[
\{ \bc_S : \bc \text{ is a maximal chain of } B_d \text{ with } \bc_{[d-1]\setminus S} = (\bd_i)_{[d-1]\setminus S}\}
\]  
Finally, let $\Gamma_i$ be the simplicial complex with facets given by maximal chains in $L_i$ that are not chains in $L_j$ for any $j < i$.  

\begin{thm}[\cite{me}]\label{booleanced}
The sequence $\Gamma_1, \Gamma_2, \ldots, \Gamma_t$ is a convex-ear decomposition of $\Delta((B_d)_S)$. 
\end{thm}

The following lemmata, whose proofs we omit, are shown in \cite{me}.

\begin{lem}\label{inc}
Let $\be$ be a maximal chain of some $L_i$.  Then $\be$ is a facet of $\Gamma_i$ if and only if 
\[
(\Upsilon_\lambda(\be))_{[d-1]\setminus S} = (\bd_i)_{[d-1]\setminus S}.
\]
\end{lem}

\begin{lem}\label{increasing}
Let $\bc$ be the unique maximal chain of $B_d$ with increasing $\lambda$-label.  Then 
\[
\Upsilon_\lambda((\bd_1)_{[d-1]\setminus S}) = \bc.
\]
\end{lem}

\begin{lem}\label{sigmashell}
Let $\be_1, \be_2, \ldots, \be_m$ be all maximal chains corresponding to facets of $\Gamma_i$, ordered so that $\lambda(\Upsilon_{\lambda}(\be_k))$ lexicographically precedes $\lambda(\Upsilon_{\lambda}(\be_j))$ whenever $j < k$.  Then for all $j$ and $k$ with $j < k$, there exists a $j' < k$ satisfying
\[
\be_j \cap \be_k \subseteq \be_{j'} \cap \be_k = \be_{j'} - x
\]
for some element $x$ of $\be_{j'}$.  
\end{lem}

Lemma \ref{sigmashell}, together with Theorem \ref{danarajklee} and Definition \ref{altshell}, proves that $\Gamma_i$ is a topological ball for $i> 1$.  We are now ready to prove our main theorem.  

\begin{proof}[Proof of Theorem \ref{boolpieces}]
The basic idea is to iterate the decomposition provided by Theorem \ref{booleanced}.  Indeed, Theorem \ref{booleanced} gives us a c.e.d. of $\Delta((P_1)_S)$.  Now suppose we have a c.e.d. for $X =\bigcup_{i = 1}^{q-1} \Delta((P_i)_S)$ for some $q$ with $2 \leq q \leq r$.  We show that we can extend this to a c.e.d. of $X \cup \Delta((P_q)_S)$.  For ease of notation, let $\lambda = \lambda_q$ and $\Upsilon = \Upsilon_\lambda$.  

Taking our cue from the decomposition of $(B_d)_S$ described above, let $\bd_1, \bd_2, \ldots, \bd_t$ be all maximal chains of $P_q$ whose $\lambda$-labels have descent set $S$, and let the above order be such that $\lambda(\bd_i)$ lexicographically precedes $\lambda(\bd_j)$ for $i < j$.  For each $i$, define $L_i$ and $\Gamma_i$ as in Theorem \ref{booleanced}.  Finally, let $\Sigma_i$ be the simplicial complex whose facets are maximal chains in $\Gamma_i$ that are not chains in $X$.  We claim that the sequence $\Sigma_1, \Sigma_2, \ldots, \Sigma_t$ (once we remove all $\Sigma_i = \emptyset$) extends the c.e.d. of $X$.  We prove each property of Definition \ref{ced} separately.  

By definition, each $\Gamma_i \subseteq \Sigma_i \cup X$.  Since $\bigcup_{i = 1}^t \Gamma_i = \Delta((P_q)_s)$ (by Theorem \ref{booleanced}), $X \cup (\bigcup_{i = 1}^t \Sigma_i ) = X \cup \Delta((P_q)_S)$, so property (\emph{i}) holds.  

Property (\emph{ii}) is easily verified as well.  Since each $\Gamma_i$ for $i > 1$ is a proper subcomplex of a simplicial polytope boundary, so is each $\Sigma_i \subseteq \Gamma_i$.  However, as $\Gamma_1$ \emph{is} a simplicial polytope boundary, we need to show that the inclusion $\Sigma_1 \subseteq \Gamma_1$ is proper.  This follows from Lemma \ref{increasing}, which says that $\bc_S$ is a facet of $\Gamma_1$, where $\bc$ is the unique maximal chain in $P_q$ with increasing $\lambda$-label.  Because $\bc = \Upsilon(\hat{0} < \hat{1})$ and $\hat{0} < \hat{1}$ is a chain in all $P_i$, it follows that $\bc_S$ is not a facet of $\Sigma_1$.

Now fix $i \geq 1$.  To prove property (\emph{iii}), we employ the techniques (and notation) of Lemma \ref{sigmashell}.  Let $\be_1, \be_2, \ldots, \be_m$ be all maximal chains of $\Sigma_i$, ordered so that $\lambda(\Upsilon(\be_k))$ precedes $\lambda(\Upsilon(\be_j))$ whenever $j < k$.  Now choose some $j$ and $k$ with $j < k$.  Because $\Sigma_i \subseteq \Gamma_i$, Lemma \ref{sigmashell} produces a maximal chain $\be$ in $\Gamma_i$ satisfying $\be_j \cap \be_k \subseteq \be \cap \be_k = \be - x$ for some element $x$ of $\be$, with $\lambda(\Upsilon(\be_k))$ lexicographically preceding $\lambda(\Upsilon(\be))$.  To finish the proof, we just need to show that $\be$ is a facet of $\Sigma_i$.  That is, we need to show that $\be$ is not a chain in $X$. 
  
By Lemma \ref{inc}, $\be \cap \be_k = \be - x$ implies that $\Upsilon(\be) \cap \Upsilon(\be_k) = \Upsilon(\be) - x$.  Because $P_q$ is a Boolean lattice, it has exactly two maximal chains containing $\Upsilon(\be) - x$ as a subchain.  Hence, these chains must be $\Upsilon(\be)$ and $\Upsilon(\be_k)$.  Because $\lambda(\Upsilon(\be_k))$ precedes $\lambda(\Upsilon(\be))$, we have $\Upsilon(\be_k) = \Upsilon(\Upsilon(\be) - x)$.  If $\be$ were a chain in $X$, then $\Upsilon(\be) - x$ would be as well.  But then, since $\be_k$ is a subchain of $\Upsilon(\Upsilon(\be) - x)$, we would have that $\be_k$ is in $X$, a contradiction.  

For property (\emph{iv}) consider some $i$, and note that a chain $\be$ is in $\partial \Sigma_i$ if and only if there exist two maximal chains $\be_\old$ and $\be_\new$, each containing $\be$ as a subchain, such that $\be_\old$ is a chain in $X \cup (\bigcup_{j = 1}^{i-1} \Sigma_j)$ and $\be_\new$ is a chain in $\Sigma_i$.  Thus, 
\[
\partial\Sigma_i \subseteq   \Sigma_i \cap (X \cup (\bigcup_{j = 1}^{i-1} \Sigma_j ) ).
\]  
To prove the reverse inclusion, let $\be$ be a non-maximal chain in $\Sigma_i \cap (X \cup (\bigcup_{j = 1}^{i-1} \Sigma_j ) )$.  Then by definition $\be$ must be a subchain of some maximal chain in $\Sigma_i$, and we can take this chain to be $\be_\new$.  To find $\be_\old$, we consider two cases.  First, if $\be$ is a chain of $\Sigma_j$ for some $j < i$, then by Theorem \ref{booleanced} there must be some maximal chain $\be_\old$ of $\Sigma_j$ for some $j < i$.  Second, if $\be$ is a chain in $X$, then $\Upsilon(\be)$ must be in $X$ as well.  Setting $\be_\old = (\Upsilon(\be))_S$ completes the proof of property (\emph{iv}).

Thus, we can extend the c.e.d. of $X$ to one of $X \cup \Delta((P_q)_S)$.  Continuing in this fashion, we get a c.e.d. of 
\[
\bigcup_{i=1}^r \Delta((P_i)_S).
\] 
By hypothesis every chain in $P$ is a chain in some $P_i$ and the above union equals $\Delta(P_S)$, proving the theorem. 
\end{proof}

\ssection{Rank-selected geometric lattices}\label{geolattice}

We first apply Theorem \ref{boolpieces} to geometric lattices.  We assume a basic familiarity with matroid theory, including the cryptomorphism between matroids and geometric lattices.  For background, see \cite{bj} or \cite{oxley}. 

Let $P$ be a rank-$d$ geometric lattice.  In \cite{ns}, Nyman and Swartz show that $\Delta(P)$ admits a convex-ear decomposition.  We open this section by briefly describing their technique.

Let $a_1, a_2, \ldots, a_\ell$ be a fixed linear ordering of the atoms of $P$.  The \emph{minimal labeling} $\nu$ of $P$ is defined as follows: if $y$ covers $x$, then $\nu(x,y)=\min\{ i : x \vee a_i=y\}$.  We view $P$ as the lattice of flats of a simple matroid $M$.    

\begin{lem}[\cite{bj}]
The minimal labeling $\nu$ is an EL-labeling.   
\end{lem}

\begin{lem}[\cite{bj}]\label{nbc}
Suppose the $\nu$-label of a maximal chain $\bc$ of $P$ is a word in some subset $B\subseteq \{a_1, a_2, \ldots, a_\ell\}$.  Then $B$ is an nbc-basis of $M$.  
\end{lem}

Now let $B_1, B_2, \ldots, B_t$ be all the nbc-bases of $M$ listed in lexicographic order.  Fix some $j\leq t$ and let $B_j=\{a_{i_1}, a_{i_2}, \ldots, a_{i_d}\}$ with $i_1  < i_2 < \cdots < i_d$.  For a permutation $\sigma\in \mathcal{S}_d$, define a maximal chain  $\textbf{c}_\sigma^j$ of $P$ by 
\[
\bc_\sigma^j := \hat{0} < a_{i_{\sigma(1)}} < (a_{i_{\sigma(1)}}\vee a_{i_{\sigma(2)}} ) < \ldots <  ( a_{i_{\sigma(1)}}\vee a_{i_{\sigma(2)}}\vee \ldots \vee a_{i_{\sigma(d)}}).
\]  
The \emph{basis labeling} $\lambda_j(\textbf{c}_\sigma^j)$ of $\textbf{c}_\sigma^j$ is the word
$i_{\sigma(1)}i_{\sigma(2)}\ldots i_{\sigma(d)}$. 

For each $i$ with $1 \leq i\leq t$, let $P_i$ be the subposet of $P$ whose set of maximal chains is $\{\textbf{c}_\sigma^i : \sigma \in \mathcal{S}_d\}$ and let $\Sigma_i$ be the simplicial complex whose facets are maximal chains in $P_i$ that are not chains in $P_j$ for any $j < i$.   

\begin{thm}[\cite{ns}]  
$\Sigma_1,\Sigma_2, \ldots,\Sigma_t$ is a convex-ear decomposition of $\Delta(P)$.  
\end{thm}
     
The next lemma, shown in \cite{ns}, is the key tool in proving the above theorem.

\begin{lem}\label{labels}
A chain $\textbf{c}$ in $P_i$ is in $\Sigma_i$ if and only if
\[
\lambda_i(\bc)=\nu(\bc).
\]
\end{lem} 

The main theorem of this section is the following:

\begin{thm}
Let $P$ be a rank-$d$ geometric lattice.  Then $\Delta(P_S)$ admits a convex-ear decomposition for any $S \subseteq [d-1]$. 
\end{thm}

\begin{proof}
We show that $P$ satisfies the hypotheses of Theorem \ref{boolpieces}, proving each of the three properties separately. 

First note that each $P_i$ is isomorphic to the Boolean lattice $B_d$ under the mapping $a_{i_{\sigma(1)}}\vee a_{i_{\sigma(2)}}\vee \ldots \vee a_{i_{\sigma(m)}}\rightarrow \{\sigma(1),\sigma(2), \ldots, \sigma(m)\}$, and so property (\emph{1}) holds.  Moreover, the basis labeling $\lambda_i$ is an $\mathcal{S}_d$-EL-labeling of $P_i$ (though with the alphabet $\{i_1, i_2, \ldots, i_d\}$ rather than $[d]$).  

By Lemma \ref{nbc}, the $\nu$-label of any maximal chain $\bc$ is a word in some nbc-basis (say $B_i$).  Thus $\bc$ is a chain in $P_i$, meaning property (\emph{2}) holds.  

To show property (\emph{3}), fix some $i$ and suppose that $\textbf{e}$ is a non-maximal chain in $P_i$ that is also a chain in $P_j$ for some $j < i$.  Suppose that $j$ is the least such integer, and consider the maximal chain $\textbf{c} = \Upsilon_{\lambda_j}(\textbf{e})$.  This chain can clearly not be in $L_k$ for any $k < j$, because then $\textbf{e}$ would be a chain in $L_k$, contradicting the minimality of $j$.  Thus $\lambda_j(\bc)=\nu(\bc)$ by Lemma \ref{labels}, meaning $\textbf{c} = \Upsilon_\nu(\be)$.  Now consider the chain $\textbf{c}' = \Upsilon_{\lambda_i}(\be)$.  If $\bc'$ is not a chain in $L_k$ for any $k < i$ then, again by Lemma \ref{labels}, $\nu(\bc')= \lambda_i(\bc')$.  But then $\bc'=\Upsilon_\nu(\be)$, which is a contradiction since the chain $\Upsilon_\nu(\be)$ is uniquely determined.  Thus $\Upsilon_{\lambda_i}(\be)$ must be a chain in $L_k$ for some $k < i$.  Applying Theorem \ref{boolpieces} completes the proof.  
\end{proof}

\ssection{Rank-selected face posets}\label{faceposetsection}

The main result of this section can be seen as motivated by Hibi's result (\cite{hi}) that the codimension-$1$ skeleton of a shellable complex $\Sigma$ is $2$-Cohen-Macaulay.  For a simplicial complex $\Sigma$ we write $P_\Sigma$ to mean its \emph{face poset}, the poset of all faces of $\Sigma$ ordered by inclusion.  Note that $P_\Sigma$ usually does not have a unique maximal element, but the notion of the rank-selected subposet $(P_\Sigma)_S$ easily generalizes. 

\begin{thm}\label{faceposet}
Let $\Sigma$ be a $(d-1)$-dimensional shellable complex.  Then $\Delta((P_\Sigma)_S)$ admits a convex-ear decomposition for any $S \subseteq [d-1]$.  
\end{thm}

\begin{proof}
We wish to apply Theorem \ref{boolpieces} but, as noted above, a slight adjustment is needed: unless $\Sigma$ consists of a single facet, $P_\Sigma$ has no maximal element.  To this end, let $P$ be the poset $P_\Sigma$ with all its maximal elements identified.  As usual, let $\hat{1}$ denote the maximal element of $P$.  Clearly, for any $S \subseteq [d-1]$, 
\[
\Delta(P_S) = \Delta((P_\Sigma)_S)
\]
So, it suffices to apply Theorem \ref{boolpieces} to $P$.  Fix a shelling $F_1, F_2, \ldots, F_r$ of $\Sigma$, and for each $i$ let $P_i$ be the face poset of $F_i$ (but with its maximal element $F_i$ replaced with $\hat{1}$, the maximal element of $P$).  We claim that the sequence $P_1, P_2, \ldots P_r$ satisfies the hypotheses of Theorem \ref{boolpieces}.  Property (\emph{1}) follows immediately, as the face poset of a $(d-1)$-dimensional simplex is isomorphic to $B_d$.  

For property (\emph{2}), let $\bc$ be a maximal chain of $P$, and let $x$ be its element of rank $d-1$.  Then $x$ is a face of some facet $F_i$, meaning $P_i$ contains the chain $\bc$.  

The proof of property (\emph{3}) relies on the following fact, whose proof is immediate.

\begin{fact}\label{easyfact}
Let $\textbf{e}$ be a non-maximal chain in some $P_i$, and let $x$ be its element of highest rank $\neq d$.  Then $\textbf{e}$ is not a chain in $P_j$ for any $j < i$ if and only if, when viewed as a face of $F_i$, $x$ contains the unique minimal new face $r(F_i)$.  
\end{fact}

Now fix some $i$, and let $V$ be the set of vertices of the facet $F_i$.  Any bijection $\phi: V \rightarrow [d]$ induces an $\mathcal{S}_d$-EL-labeling $\lambda_\phi$ of $P_i$ in the obvious way:  For $x,y \in P_i$ with $y = x \cup \{v\}$ for some vertex $v$ of $F_i$, set $\lambda_\phi(x,y)=\phi(v)$ (if $y = \hat{1}$, let $\lambda_\phi(x, y)$ be the sole vertex in $V \setminus x$).  Let $\phi :F_i\rightarrow [d]$ be any bijection that labels vertices in $r(F_i)$ last.  That is, if $v\in r(F_i)$ and $w\in F_i\setminus r(F_i)$ then $\phi(w) < \phi(v)$.  Set $\lambda_i = \lambda_\phi$.  Now suppose $\be$ is a non-maximal chain in $P_i$ that is also a chain in $P_j$ for some $j < i$, and let $x$ be the element of $\be$ of highest rank $\neq d$.  By Fact \ref{easyfact}, $r(F_i) \nsubseteq x$.  If $v$ is the vertex in $F_i\setminus x$ with the greatest $\phi$-label then, by definition of $\phi$, $v\in r(F_i)$.  Letting $y$ be the element of $\Upsilon_{\lambda_i}(\be)$ of rank $d-1$, it follows that $v \notin y$.  So $\Upsilon_{\lambda_i}(\be)$ is a chain in $P_k$ for some $k < i$, and property (\emph{3}) holds.
\end{proof}

In many cases, the above theorem does not hold if $d\in S$.  For example, if $\Sigma$ is the shellable complex consisting of two $2$-dimensional simplices joined at a common boundary facet and $S =\{2,3\}\subseteq [3]$, then $\Delta((P_\Sigma)_S)$ does not admit a c.e.d., as it is a tree.

\begin{example}
Let $\Sigma$ be the $2$-skeleton of two $3$-dimensional simplices joined at a common boundary facet.  By Theorem \ref{faceposet}, $\Delta((P_\Sigma)_S)$ admits a c.e.d. for any $S \subseteq [2]$.  Note, however, that $\Sigma$ admits a c.e.d. and, moreover, so does the complex $\Delta((P_\Sigma)_S)$ for any $S \subseteq [3]$.  Figure \ref{wtf} shows the case when $S = \{2,3\}$. 
\end{example}

\begin{figure}[htp]
\centering
\includegraphics[scale=.65]{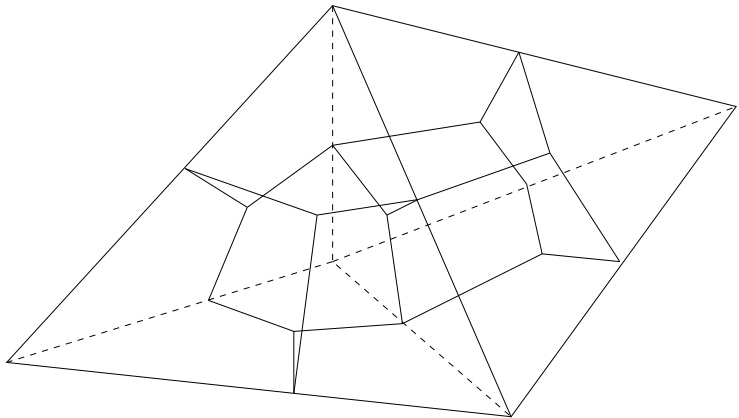}
\caption{The complex $\Delta((P_\Sigma)_{\{2,3\}})$ as a subcomplex of $\Sigma$.}\label{wtf}
\end{figure}

\begin{conj}
When $\Sigma$ is a $(d-1)$-dimensional complex admitting a convex-ear decomposition and $S\subseteq [d]$, the complex $\Delta((P_\Sigma)_S)$ admits a convex-ear decomposition.
\end{conj}

If $d\notin S$, the above conjecture follows from Theorem \ref{faceposet}, so we need only consider cases where $d\in S$.

Now recall that a $(d-1)$-dimensional complex $\Sigma$ with vertex set $V$ is called \emph{balanced} if there exists a $\psi: V\rightarrow [d]$ such that $\psi(v)\neq \psi(w)$ whenever $v$ and $w$ are in a common face of $\Sigma$.  The function $\psi$ is called a \emph{coloring} of $\Sigma$.  

The order complex of any graded poset $P$ is always balanced:  For a vertex $v$ of $\Delta(P)$, simply let $\psi(v)$ be the rank of $v$ when considered as an element of the poset $P$.  Thus the barycentric subdivision of any simplicial complex is balanced, since it is the order complex of its face poset.  

If $\Sigma$ is a $(d-1)$-dimensional balanced complex with coloring $\psi$ and $S\subseteq [d]$, define $\Sigma_S$ to be the subcomplex of $\Sigma$ with faces $\{F\in \Sigma : \psi(v)\in S$ for all $v\in F\}$.  With these new definitions, we can rephrase Theorem \ref{faceposet} in a more geometric tone.  

\begin{cor}
Let $\Sigma'$ be a $(d-1)$-dimensional shellable complex, and let $\Sigma$ be the first barycentric subdivision of its codimension-$1$ skeleton.  Then, for any coloring $\psi$ of the vertices of $\Sigma$ and any $S\subseteq [d-1]$, the complex $\Sigma_S$ admits a convex-ear decomposition.
\end{cor}

\ssection{The flag h-vector of a face poset}\label{flaghsection}

Our goal in this section is to prove an analogue of the following theorem, shown in \cite{ns}, for Cohen-Macaulay complexes.  In this section, we assume a basic working knowledge of the Stanley-Reisner ring of a simplicial complex and its Hilbert series (see \cite{greenstanley}).  

\begin{thm}[\cite{ns}]
Let $L$ be a rank-$d$ geometric lattice, and let $S, T \subseteq [d-1]$.  If $S$ dominates $T$, the flag h-vector of $\Delta(L)$ satisfies $h_T \leq h_S$.
\end{thm}

We state the main theorem now, but postpone its proof until after Theorem \ref{flagh}.  

\begin{thm}\label{cmflag}  
Let $K$ be a $d$-dimensional Cohen-Macaulay simplicial complex with face poset $P$, and let $\Delta=\Delta(P)$.  Let $S,T\subseteq [d-1]$, and suppose that $S$ dominates $T$.  Then the flag h-vector of $\Delta$ satisfies $h_T\leq h_S$.
\end{thm}

\begin{lem}\label{flagstuff}
Let $P$ be a rank-$d$ poset with a $\hat{0}$ and $\hat{1}$ whose order complex $\Delta = \Delta(P)$ is a ball.  Let $\Delta'$ be the set of faces in $\Delta - \partial \Delta - \emptyset$, let $\{f'_S\}$ be the flag f-vector of $\Delta'$, and let $\{h_S\}$ be the flag h-vector of $\Delta$.  Then
\[
\sum_{S\subseteq[d-1]}f_S'\prod_{i\notin S}(\nu_i-1)=\sum_{S\subseteq[d-1]}h_{[d-1]- S}\prod_{i\notin S}\nu_i
\]
\end{lem}

\begin{proof}
Under the fine grading of the face ring $k[\Delta]$, $F(k[\Delta],\lambda)=\sum_{F\in \Delta}\prod_{x_i\in F}\frac{\lambda_i}{1-\lambda_i}$.  We specialize this grading to accommodate the flag h-vector as follows:  identify $\lambda_i$ and $\lambda_j$ whenever the vertices in $\Delta$ to which they correspond have the same rank $r$ (as elements of $P$).  Call this new variable $\nu_r$.  This specialized grading yields: 
\[
F(k[\Delta],\nu)=\sum_{S\subseteq [d-1]}  f_S\prod_{i\in S}\frac{\nu_i}{1-\nu_i}
\]
We put this over the common denominator of $\prod_{i\in [d-1]}(1-\nu_i)$ to obtain:  
\begin{equation}\label{eq}
F(k[\Delta],\nu)=\sum_{S\subseteq [d-1]}\frac{f_S \prod_{i\in S}\nu_i\prod_{i\notin S}(1-\nu_i)}{\prod_{i\in [d-1]}(1-\nu_i)}=\sum_{S\subseteq [d-1]}\frac{h_S\prod_{i\in S}\nu_i}{\prod_{i\in[d-1]}(1-\nu_i)}
\end{equation}
The following equation is Corollary II.7.2 from \cite{greenstanley} (note that $\Delta$ is $(d-2)$-dimensional):
\[
(-1)^{d-1}F(k[\Delta],1/\lambda)=(-1)^{d-2}\tilde{\chi}(\Delta)+\sum_{F\in \Delta'}\prod_{x_i\in F}\frac{\lambda_i}{1-\lambda_i}
\]
Noting that $\tilde{\chi}(\Delta)=0$, plugging in $1/\lambda$ in place of $\lambda$, and specializing to the $\nu$-grading, the previous expression becomes:
\[
(-1)^{d-1}F(k[\Delta],\nu)=\sum_{S\subseteq [d-1]}f_S'\prod_{i\in S}\frac{1}{\nu_i-1}
\]
Putting the above over the common denominator of $\prod_{i\in [d-1]}(\nu_i-1)$ and multiplying by $(-1)^{d-1}$ gives us:
\[
F(k[\Delta],\nu)=\sum_{S\subseteq [d-1]}\frac{f_S'\prod_{i\notin S}(\nu_i-1)}{\prod_{i \in [d-1]}(1-\nu_i)}
\]
Comparing this with Equation \eqref{eq} and noting that the denominators are equal, we have
\[
\sum_{S\subseteq[d-1]}f_S'\prod_{i\notin S}(\nu_i-1) = \sum_{S\subseteq[d-1]}h_S\prod_{i\in S}\nu_i= \sum_{S\subseteq[d-1]}h_{[d-1]- S}\prod_{i\notin S}\nu_i
\]
which proves the result.
\end{proof}

Now let $\Sigma$ be a $(d-1)$-dimensional shellable complex with face poset $P_\Sigma$ and shelling order $F_1, F_2, \ldots,  F_t$, and for each $i$ let $P_i$ be the face poset of $F_i$.  Let $A=[d-1]$, and set $\Delta = \Delta((P_\Sigma)_A)$.  Note that $\Delta$ is simply the order complex of $P_\Sigma$ once we remove the elements corresponding to the facets of $\Sigma$.  Let $\Sigma_1,\Sigma_2,\ldots,\Sigma_t$ be the c.e.d. of $\Delta$ given by Theorem \ref{faceposet}.  

\begin{lem}\label{switch}  
Fix some $\Sigma_i$, and let $\lambda_i$ be the labeling of $P_i$ constructed in the proof of Theorem \ref{faceposet}.  Let $S,T\subseteq [d-1]$, and suppose $S$ dominates $T$.  Then there are at least as many maximal chains in $\Sigma_i$ whose $\lambda_i$-labels have descent set $S$ as chains whose $\lambda_i$-labels have descent set $T$.  
\end{lem} 

\begin{proof}
Let $\bc$ be a maximal chain in $\Sigma_i$ with $\sigma = \lambda_i(\bc)$, and let $\sigma(j) < \sigma(j+1)$ be an ascent of its $\lambda_i$-label.  Let $\bc'$ is the unique maximal chain of $P_i$ that coincides with $\bc$ at every rank but $j$, and let $x$ be the element of $\bc'$ of rank $j$.  Then $\bc'$ must be a chain in $\Sigma_i$, since otherwise $\bc = \Upsilon_{\lambda_i}(\bc' - x)$ would not be a chain in $\Sigma_i$.  So if $\tau$ is a permutation preceded by $\sigma$ in the weak order, there is some maximal chain in $\Sigma_i$ with $\tau$ as its $\lambda_i$-label.  Now suppose $\sigma$ has descent set $T$.  If $\phi$ is as in Definition \ref{dominance}, it follows that $\Sigma_i$ contains a chain whose $\lambda_i$-label is $\phi(\sigma)$.
\end{proof}

\begin{thm} \label{flagh}
Let $\Sigma$ and $\Delta$ be as above, let $S,T\subseteq[d-1]$, and suppose that $S$ dominates $T$.  Then the flag h-vector of $\Delta$ satisfies $h_T\leq h_S$.
\end{thm}

\begin{proof}
First, note that $h_T(\Sigma_1)\leq h_S(\Sigma_1)$, since the poset associated to $\Sigma_1$ is just the Boolean lattice $B_d$.  Now let $\Omega=\Sigma_1\cup\Sigma_2 \cup \cdots \cup \Sigma_{k-1}$ and suppose the result holds for $\Omega$.  Let $\Sigma_k'=\Sigma_k-\partial\Sigma_k-\emptyset$.  Because $\Sigma_k$ triangulates a ball, we can now use our earlier expression for the flag h-vector of a ball and invoke an argument similar to Chari's in \cite{ch}:
\begin{align*} 
\sum_{S\subseteq [d-1]}h_S(\Omega\cup\Sigma_k)\prod_{i\notin S}\nu_i&=\sum_{S\subseteq [d-1]}f_S(\Omega\cup\Sigma_k)\prod_{i\notin S}(\nu_i-1) \\
&=\sum_{S\subseteq [d-1]}f_S(\Omega)\prod_{i\notin S}(\nu_i-1)+\sum_{S\subseteq [d-1]}f_S(\Sigma_k')\prod_{i\notin S}(\nu_i-1)\\
&=\sum_{S\subseteq [d-1]}h_S(\Omega)\prod_{i\notin S}\nu_i+\sum_{S\subseteq [d-1]}h_{[d-1]-S}(\Sigma_k)\prod_{i\notin S}\nu_i \text{ (by Lemma \ref{flagstuff})}\\
&=\sum_{S\subseteq [d-1]}(h_S(\Omega)+h_{[d-1]-S}(\Sigma_k))\prod_{i\notin S}\nu_i
\end{align*}
So for all subsets $S \subseteq [d-1]$, 
\[
h_S(\Omega \cup \Sigma_k) =  h_S(\Omega) + h_{[d-1]-S}(\Sigma_k)
\]

Because reverse lexicographic order of the maximal chains of $\Sigma_k$ is a shelling (Theorem \ref{boolpieces}), $h_S(\Sigma_k)$ counts the number of maximal chains of $\Sigma_k$ whose labels have \emph{ascent} set $S$, and so $h_{[d-1]-S}(\Sigma_k)$ counts the number of maximal chains in $P_k$ with descent set $S$.  The result now follows from Lemma \ref{switch}.
\end{proof}

\begin{proof}[Proof of Theorem \ref{cmflag}]
Because $P$ is a face poset, a linear inequality of its flag h-vector translates to a linear inequality of the h-vector of $K$.  Since every h-vector of a Cohen-Macaulay complex is the h-vector of some shellable complex (see \cite{greenstanley}), the result follows.
\end{proof}

We now show that Theorem \ref{cmflag} cannot be extended to include posets whose order complexes are Cohen-Macaulay (or 2-CM, for that matter).  First recall that a graded poset $P$ is \emph{Eulerian} if its M\"obius function satisfies $\mu(x,y)=(-1)^{\text{rank}(y)- \text{rank}(x)}$ for all $x < y$.  An Eulerian poset whose order complex is Cohen-Macaulay is called \emph{Gorenstein*}.  It can be shown that the order complex of a Gorenstein* poset is 2-Cohen-Macaulay.  For $S\subseteq [d-1]$, define $w(S)$ to be the set of all $i\in [n]$ such that exactly one of $i$ and $i+1$ is in $S$.  For instance, if $S=\{2,3\}\subseteq [4]$ then $w(S)=\{1,3\}$.  Since Conjecture 2.3 from \cite{st4} was proven by Karu in \cite{ka}, we can rephrase Proposition 2.8 from \cite{st4} as:

\begin{prop}[\cite{st4}]  
If $S,T\subseteq [d-1]$ are such that $h_T(\Delta)\leq h_S(\Delta)$ whenever $\Delta$ is the order complex of a Gorenstein* poset, then $w(T)\subseteq w(S)$.  
\end{prop}

Now consider $S,T\subseteq [4]$ given by $S=\{1,2\}$ and $T=\{1\}$.  In \cite{ns}, it is shown that $S$ dominates $T$.  However, $w(S)=\{2\}$ and $w(T)=\{1\}$, so $w(T)\nsubseteq w(S)$ and it is clear that we cannot extend Theorem \ref{cmflag} to include the wider class of Cohen-Macaulay posets (or even 2-CM posets). \\

\textbf{Acknowledgement:} This paper, which constituted part of the author's Ph.D. thesis, could not have been written without the encouragement, input, and patience of Ed Swartz.  

\bibliography{mybib2}{}
\bibliographystyle{plain}

\end{document}